\documentclass{amsart}
\usepackage{graphicx}
\usepackage{amssymb}
\usepackage{amsfonts}
\usepackage{mathrsfs}
\usepackage{xcolor}

  [1995/08/01 v0.1 Double stroke roman fonts]

\def\ds@whichfont{dsrom}
\relax

\DeclareMathAlphabet{\mathds}{U}{\ds@whichfont}{m}{n}

\swapnumbers
\newtheorem{theorem}{Theorem}[section]
\newtheorem{lemma}[theorem]{Lemma}
\newtheorem{corollary}[theorem]{Corollary}
\newtheorem{example}[theorem]{Example}

\theoremstyle{definition}

\newtheorem{remark}[theorem]{Remark}

\numberwithin{equation}{section}
\theoremstyle{plain}

\numberwithin{equation}{section} 
\numberwithin{figure}{section} 
\theoremstyle{plain}
\theoremstyle{plain}
\theoremstyle{remark}
\newtheorem*{acknowledgement*}{Acknowledgement}
\DeclareMathOperator{\Ima}{Im}

\newcommand{\cA}{{\mathcal A}}

\newcommand{\bbN}{{\mathbb N}}

\newcommand{\bbI}{{\mathbb I}}

\begin{document}
\title[A general approach to nonautonomous shadowing]{A general approach to nonautonomous shadowing for nonlinear dynamics}

\author{Lucas Backes}
\address{\noindent Departamento de Matem\'atica, Universidade Federal do Rio Grande do Sul, Av. Bento Gon\c{c}alves 9500, CEP 91509-900, Porto Alegre, RS, Brazil.}
\email{lucas.backes@ufrgs.br}

\author{Davor Dragi\v cevi\' c}
\address{Department of Mathematics, University of Rijeka, Croatia}
\email[Davor Dragi\v cevi\' c]{ddragicevic@math.uniri.hr}

\keywords{shadowing, nonautonomous systems, nonlinear dynamics, dichotomies}
\subjclass[2010]{Primary: 37C50, 34D09; Secondary: 34D10}

\maketitle

\begin{abstract}
Given a nonautonomous and nonlinear differential equation
\begin{equation}\label{DE}
x'=A(t)x+f(t,x) \quad t\geq 0,
\end{equation}
on an arbitrary Banach space $X$, we formulate  very general conditions for the associated linear equation $x'=A(t)x$ and for the nonlinear term $f:[0,+\infty)\times X\to X$ under which the above system satisfies an appropriate  version of the shadowing property. More precisely, we require that $x'=A(t)x$ admits a very general type of dichotomy,  which includes the classical hyperbolic behaviour as a very particular case. In addition, we require that $f$ is Lipschitz in the second variable with a sufficiently small Lipschitz constant. Our general framework enables us to treat various settings in which no shadowing result has been previously obtained. Moreover, we are able  to recover and refine several known results. 
 We also show how our main results can be applied to the study of the shadowing property for  higher order differential equations. Finally, we conclude the paper by presenting a discrete time versions of our results.
\end{abstract}

\section{Introduction}

It is a very well known fact that, in general, it might be very complicated or even impossible to explicitly find a solution of a given differential equation. Nevertheless, such complicated equations may be models to some important classes of  dynamics arising naturally in science and thus it is necessary to study them. A possible approach to the  study this kind of systems is the following: instead of looking for exact solutions of these differential equations,  we could look for their approximate solutions.

Over the last decades, computational power has reached a high state of development allowing one to construct very good approximate solutions of a given differential equation over large intervals of time. However, the information given by these approximate solutions will be useful only if they correspond to actual solutions of the original equation, i.e. only if there exists a true orbit of the system with a slightly different initial condition that stays near the approximate trajectory. The systems exhibiting this property are said to have the \emph{shadowing property}.

The aim of the present paper is to present general conditions under which a nonautonomous system has the shadowing property. More precisely, given a nonautonomous and nonlinear  differential equation
\begin{equation}\label{eq: eq intro}
x'=A(t)x+f(t,x) \quad t\geq 0,
\end{equation} 
on an arbitrary Banach space $X$, we  present very general conditions  for the associated linear equation $x'=A(t)x$  and for the nonlinear term  $f:[0,+\infty)\times X\to X$ under which \eqref{eq: eq intro} has a version of the shadowing property. 

Following our recent paper~\cite{BD2} (which in turn is inspired by recent works on the so-called quasi-shadowing property in the context of smooth dynamics~\cite{CRV, HZZ, LZ, KT, ZZ}), we assume that $X$ splits into three directions: the stable, the unstable and the central direction. All three directions are 
invariant under the action of the dynamics associated with the linear equation $x'=A(t)x$ and finally, when restricted to the sum of the stable and the unstable direction, the dynamics admits the so-called $(\mu, \nu)$-dichotomy. Then, if $f$ is Lipschitz in the second variable with a sufficiently small Lipschitz constant, we prove that  close  to each  approximate solution $t\mapsto y(t)$ of~\eqref{DE}, one can find a map $x\mapsto x(t)$ which is a solution of~\eqref{DE} up to the movement along the central direction. In the particular case when the central direction vanishes, our result says that in a vicinity of each approximate solution of~\eqref{DE}, we can find its 
exact solution. In other words, \eqref{DE} has the shadowing property. 

The importance of our results stems from our general framework. As we already emphasized,  we assume that the linear equation $x'=A(t)x$ (when restricted to the sum of the stable and the unstable direction) admits the so-called $(\mu, \nu)$-dichotomy. This notion was 
introduced and studied by Bento and Silva~\cite{BS, BS2} and includes the classical notion of  an exponential dichotomy~\cite{C} as a very particular case. In addition, it also extends the notion of a polynomial dichotomy~\cite{BS0, BS1} as well as the notion of a generalized dichotomy as introduced by Muldowney~\cite{M}. 
It also includes several classes of asymptotic behaviour that haven't been studied previously. We refer to~\cite{BS, BS2} for a detailed  discussion.

Hence, by adapting this framework, we are able to simultaneously recover and refine several known results as well as to discuss the shadowing property in various settings in which no such results were previously available. 
 When $x'=A(t)x$ admits an exponential or a generalized exponential dichotomy, we  basically recover our previous results obtained in~\cite{BD2, BD1}. However, to the best of our knowledge all previously known results (either for continuous or discrete time) related to the shadowing of~\eqref{DE} (or its discrete time version) belong to this category. Indeed, 
the case when the nonlinear part vanishes and when the linear part is either constant or periodic was treated in~\cite{BBT,BCDMP, BLR, BRST}, while in our previous works~\cite{BD0,BD2, BD1} we dealt with the general case. 
On the other hand, there are no previous  shadowing-type results in the setting when $x'=A(t)x$ admits a $(\mu, \nu)$-dichotomy (or even a polynomial dichotomy for example). 

Consequently, our paper aims to open a new direction of research. Indeed, while in the context of the  smooth dynamics shadowing is related to some form of hyperbolicity~\cite{P,Pil}, we show that in the setting of nonautonomous dynamics it is possible to obtain shadowing-type results for the classes of dynamics which don't
exhibit hyperbolic behaviour.

Our techniques build on our previous works~\cite{BD2, BD1}, which in turn are inspired by analytic proofs of the classical shadowing lemma~\cite{MS, P1}. An important feature of the present paper is that in comparison with our previous works, we deal with the case of continuous time directly,  without passing to the associated discrete
time.

The paper is organized as follows. In Section \ref{sec: prelim} we introduce the setting we are going to work on and the conditions on $x'=A(t)x$ and $f$ that we are going to work with. In Section \ref{sec: main} we present the main result of the paper along with its proof, while in Section \ref{sec: examples} we present examples of settings to which our results are applicable as well as their relation with the results already available in the literature. Section \ref{sec: higher order} is dedicated to the  applications of our result to higher order differential equations while in Section \ref{sec: discrete} a discrete time version of our main result is presented.

\section{Preliminaries}\label{sec: prelim}

Let $X=(X, |\cdot |)$ be an arbitrary Banach space. By $\mathcal B(X)$ we will denote the space of all bounded linear operators on $X$. The operator norm on $\mathcal B(X)$ will be denoted by $\|\cdot \|$.

For a continuous map $A\colon [0, +\infty)\to \mathcal B(X)$, we  consider a linear nonautonomous equation given by
\begin{equation}\label{lde}
x'=A(t)x, \quad t\ge 0.
\end{equation}
By $T(t,s)$, $t,s\ge 0$ we will denote the evolution family associated to~\eqref{lde}. We recall that for fixed $s$ and $v\in X$, $t\mapsto T(t,s)v$ is a solution of~\eqref{lde} such that $x(s)=v$. Then, we have the following properties:
\begin{itemize}
\item $T(t,t)=\bbI$ for $t\ge 0$, where $\bbI$ denotes the identity operator on $X$;
\item $T(t,s)T(s, r)=T(t,r)$ for $t,s, r\ge 0$.
\end{itemize}

For a continuous map  $f\colon [0, +\infty)\times  X\to X$, we consider a semilinear equation given by
\begin{equation}\label{nde}
x'=A(t)x+f(t,x), \quad t\ge 0.
\end{equation}

Let us now introduce some standing assumptions. Let $\mu \colon [0, +\infty)\to \mathbb R$ be a strictly  increasing and  differentiable map such that \[\mu(0)=1 \quad  \text{and} \quad \lim_{t\to +\infty}\mu(t)=+\infty.\] In addition, let $\nu \colon [0, +\infty)\to [1,+\infty)$ be an arbitrary map. We assume the following conditions:
\begin{enumerate}
\item  there exist three families of projections $P^i(t)$, $t\ge 0$, $i\in \{1, 2, 3\}$ such that 
\[
P^1(t)+P^2(t)+P^3(t)=\bbI, \quad t\ge 0;
\]
\item for $i, j\in \{1, 2,3 \}$, $i\neq j$ and $t\ge 0$,
\[
P^i(t)P^j(t)=0;
\]
\item for arbitrary  $t,s\ge 0$ and $i\in \{1, 2, 3\}$, we have that 
\[
T(t,s)P^i(s)=P^i(t)T(t,s);
\]
\item there exist $D, \lambda >0$ and $d\ge 0$  such that 
\begin{equation}\label{d1}
\|T(t,s)P^1(s)\| \le D \bigg (\frac{\mu(t)}{\mu (s)} \bigg )^{-\lambda} \nu(s)^d \quad \text{for $t\ge s$,}
\end{equation}
and 
\begin{equation}\label{d2}
\|T(t,s)P^2(s)\| \le D \bigg (\frac{\mu (s)}{\mu (t)} \bigg )^{-\lambda}\nu(s)^d  \quad \text{for $t\le s$;}
\end{equation}
\item there exist $c\ge 0$ such that 
\begin{equation}\label{non}
|f(t,x)-f(t,y)| \le \frac{c\mu'(t)}{\mu(t) \nu (t)^d}|x-y|, \quad \text{for $t\ge 0$ and $x, y\in X$.}
\end{equation}
\end{enumerate}

For $t\ge 0$, set 
\[
E^s(t)=\Ima P^1(t), \ E^u(t)=\Ima P^2(t) \ \text{and} \ E^c(t)=\Ima P^3(t).
\]
Furthermore, we set
\[
|x|_t:= \sup_{s\ge t} \bigg (|T(s,t)x|  \bigg (\frac{\mu(s)}{\mu (t)} \bigg )^{\lambda}   \bigg) \quad \text{for $t\ge 0$ and $x\in E^s(t)$,}
\]
\[
|x|_t:= \sup_{s\le t} \bigg (|T(s,t)x| \bigg (\frac{\mu (t)}{\mu (s)} \bigg )^{\lambda} \bigg) \quad \text{for $t\ge 0$ and $x\in E^u(t)$,}
\]
and
\[
|x|_t:= \frac{\mu(t)}{\mu'(t)}|x|, \quad \text{for $t\ge 0$ and $x\in E^c(t)$.}
\]
Finally, for an arbitrary $x\in X$ let 
\[
|x|_t:= |P^1(t)x|_t+|P^2(t)x|_t+|P^3(t)x|_t.
\]
It is easy to see that $|\cdot |_t$ is a norm on $X$ for each $t\ge 0$.
We will need the following simple auxiliary  result. 
\begin{lemma} \label{lemma: comparation norms}
The following properties hold:
\begin{itemize}
\item for $t\ge 0$ and $x\in E^s(t)\cup E^u(t)$, 
\begin{equation}\label{ln}
|x|\le |x|_t ;
\end{equation}
\item for $t\ge s$ and $x\in X$, 
\begin{equation}\label{ln1}
|T(t,s)P^1(s)x|_t \le D \bigg (\frac{\mu(t)}{\mu (s)} \bigg )^{-\lambda} \nu(s)^d|x|;
\end{equation}
\item for $t\le s$ and $x\in X$, 
\begin{equation}\label{ln2}
|T(t,s)P^2(s)x|_t \le D \bigg (\frac{\mu (s)}{\mu (t)} \bigg )^{-\lambda}\nu(s)^d|x|.
\end{equation}
\end{itemize}
\end{lemma}

\begin{proof}
Observe that~\eqref{ln} follows trivially. Let us now prove~\eqref{ln1}. Using~\eqref{d1},  we have that
\[
\begin{split}
|T(t,s)P^1(s)x|_t &=\sup_{r\ge t} \bigg (|T(r,s)P^1(s)x|  \bigg (\frac{\mu(r)}{\mu (t)} \bigg )^{\lambda}     \bigg) \\
&= \bigg (\frac{\mu(t)}{\mu (s)} \bigg )^{-\lambda}   \sup_{r\ge t}\bigg  (|T(r,s)P^1(s)x| \bigg (\frac{\mu(r)}{\mu(s)} \bigg )^{\lambda}  \bigg) \\
&\le \bigg (\frac{\mu(t)}{\mu (s)} \bigg )^{-\lambda}   \sup_{r\ge s}\bigg  (|T(r,s)P^1(s)x| \bigg (\frac{\mu(r)}{\mu(s)} \bigg )^{\lambda}  \bigg) \\
&\le D \bigg (\frac{\mu(t)}{\mu (s)} \bigg )^{-\lambda} \nu(s)^d |x|.
\end{split} 
\]
Similarly, one can establish~\eqref{ln2}.
\end{proof}

\section{Main result}\label{sec: main}
We are now in a position to establish our main result. 

\begin{theorem}\label{thm1}
Assume that 
\begin{equation}\label{C}
q:=c(2D+1)+2cD\lambda^{-1} <1.
\end{equation}
Then, there exists $C>0$ with the property that for any $\delta >0$ and an arbitrary differentiable map $y\colon [0, +\infty) \to X$ satisfying
\begin{equation}\label{pseudo}
|y'(t)-A(t)y(t)-f(t, y(t))| \le \frac{\delta \mu'(t)}{\mu(t) \nu(t)^d} \quad \text{for $t\ge 0$,}
\end{equation}
there exists a differentiable map $x\colon [0, +\infty)\to X$ such that:
\begin{itemize}
\item $P^3(t)(x(t)-y(t))=0$ for $t\ge 0$;
\item $\sup_{t\ge 0} |x(t)-y(t)| \le C\delta$;
\item $x'(t)-A(t)x(t)-f(t, x(t))\in \Ima P^3(t)$ for $t\ge 0$;
\item for $t\ge 0$,
\begin{equation}\label{EST}
|x'(t)-A(t)x(t)-f(t, x(t))| \le C\delta(2D+1)\nu (t)^d \max \bigg \{ 1, \frac{\mu'(t)}{\mu(t)} \bigg \}.
\end{equation}
\end{itemize}
\end{theorem}

\begin{remark}
We observe that this result may be seen as a version of the shadowing property. Indeed, Theorem \ref{thm1} says that given a ``pseudo-trajectory" $t\mapsto y(t)$ of \eqref{nde}, there exists a map $t\mapsto x(t)$ which is a solution of \eqref{nde} up to moving it in the central direction that ``shadows" $t\mapsto y(t)$ and agrees with it in the central direction $E^c(t)$. In particular, whenever $P^3(t)=0$,  the map $t\mapsto x(t)$ is an actual solution of \eqref{nde} and we get the usual shadowing property (see Corollary~\ref{cor}). Moreover, if $P^3(t)\neq 0$ then one cannot expect that, in general, the map $t\mapsto x(t)$ is a solution of \eqref{nde} (see \cite{BDT}) and in this case condition \eqref{EST} gives us an estimate for the deviation of $x(t)$ from being such a solution.
\end{remark}

\begin{proof}[Proof of Theorem~\ref{thm1}]
Let
\[
\mathcal Y:=\bigg \{ z\colon [0, +\infty) \to X:  \text{$z$ is continuous and} \ \|z\|_\infty:=\sup_{t\ge 0} |z(t)|_t <+\infty \bigg \}.
\]
Then, $(\mathcal Y, \|\cdot \|_\infty)$ is a Banach space. For $z\in \mathcal Y$, we (formally) set
\[
\begin{split}
(\mathcal Tz)(t) &=-P^3(t)(A(t)y(t)+f(t, y(t)+(\bbI-P^3(t))z(t))-y'(t)) \\
&\phantom{=}+\int_0^t T(t,s)P^1(s)(A(s)y(s)+f(s, y(s)+(\bbI-P^3(s))z(s))-y'(s))\, ds \\
&\phantom{=}-\int_t^\infty T(t,s)P^2(s)(A(s)y(s)+f(s, y(s)+(\bbI-P^3(s))z(s))-y'(s))\, ds,
\end{split}
\]
$t\ge 0$. Observe that it follows from~\eqref{ln1} and~\eqref{pseudo} that 
\[
\begin{split}
 & \bigg |\int_0^t T(t,s)P^1(s)(A(s)y(s)+f(s, y(s))-y'(s))\, ds \bigg |_t \\
&\le \int _0^t |T(t,s)P^1(s)(A(s)y(s)+f(s, y(s))-y'(s))|_t\, ds \\
&\le D\int_0^t \bigg (\frac{\mu(t)}{\mu (s)} \bigg )^{-\lambda} \nu(s)^d|A(s)y(s)+f(s, y(s))-y'(s)|\, ds \\
&\le D\delta \int_0^t \bigg (\frac{\mu(t)}{\mu (s)} \bigg )^{-\lambda} \frac{\mu'(s)}{\mu(s)} \, ds \\
&=\frac{D\delta}{\lambda} \bigg (1-\frac{1}{ \mu(t)^\lambda} \bigg ),
\end{split}
\]
and thus
\begin{equation}\label{01}
\sup_{t\ge 0} \bigg |\int_0^t T(t,s)P^1(s)(A(s)y(s)+f(s, y(s))-y'(s))\, ds \bigg |_t  \le \frac{D\delta}{\lambda}.
\end{equation}

Similarly, using~\eqref{ln2} and~\eqref{pseudo} we have that 
\[
\begin{split}
& \bigg |\int_t^\infty T(t,s)P^2(s)(A(s)y(s)+f(s, y(s))-y'(s))\, ds \bigg |_t \\
&\le D \int_t^\infty \bigg (\frac{\mu(s)}{\mu (t)} \bigg )^{-\lambda} \nu(s)^d |A(s)y(s)+f(s, y(s))-y'(s) |\, ds \\
&\le D\delta \int_t^\infty \bigg (\frac{\mu(s)}{\mu (t)} \bigg )^{-\lambda} \frac{\mu'(s)}{\mu(s)} \, ds \\
&=\frac{D\delta}{\lambda},
\end{split}
\]
and consequently 
\begin{equation}\label{02}
\sup_{t\ge 0}  \bigg |\int_t^\infty T(t,s)P^2(s)(A(s)y(s)+f(s, y(s))-y'(s))\, ds \bigg |_t \le \frac{D\delta}{\lambda}.
\end{equation}
In addition, we have 
\begin{equation}\label{aux}
\begin{split}
& |P^3(t)(A(t)y(t)+f(t, y(t))-y'(t)) |_t \\
&=\frac{\mu(t)}{\mu'(t)} |P^3(t)(A(t)y(t)+f(t, y(t))-y'(t)) | \\
&\le \frac{\mu(t)}{\mu'(t)} (2D+1)\nu(t)^d |A(t)y(t)+f(t, y(t))-y'(t)|,
\end{split}
\end{equation}
and therefore 
\begin{equation}\label{03}
\sup_{t\ge 0} |P^3(t)(A(t)y(t)+f(t, y(t))-y'(t)) |_t  \le \delta (2D+1).
\end{equation}
Observe that in~\eqref{aux},  we have used (see~\eqref{d1} and~\eqref{d2}) that 
\[
|P^3(t)x| \le |x|+|P^1(t)x|+|P^2(t)x|  \le |x|+2D\nu(t)^d |x| \le (2D+1)\nu(t)^d |x|,
\]
for $t\ge 0$ and $x\in X$. From~\eqref{01}, \eqref{02} and~\eqref{03}, we conclude  that
\begin{equation}\label{T0}
\|\mathcal T0\|_\infty \le \bar D\delta,
\end{equation}
where
\[
\bar D:=2D+2D\lambda^{-1}+1>0.
\]

Take now arbitrary $z_1, z_2\in \mathcal Y$. By~\eqref{non}, \eqref{ln} and~\eqref{ln1}, we have that 
\[
\begin{split}
& \bigg |\int_0^t T(t,s)P^1(s)(f(s, y(s)+(\bbI-P^3(s))z_1(s))-f(s, y(s)+(\bbI-P^3(s))z_2(s)))\, ds \bigg |_t \\
&\le D \int_0^t \bigg (\frac{\mu(t)}{\mu(s)} \bigg )^{-\lambda}\nu(s)^d |f(s, y(s)+(\bbI-P^3(s))z_1(s))-f(s, y(s)+(\bbI-P^3(s))z_2(s))| \, ds \\
&\le Dc\int_0^t \bigg (\frac{\mu(t)}{\mu(s)} \bigg )^{-\lambda} \frac{\mu'(s)}{\mu(s)} |(\bbI-P^3(s))z_1(s)-(\bbI-P^3(s))z_2(s)|\, ds \\
&\le  Dc\int_0^t \bigg (\frac{\mu(t)}{\mu(s)} \bigg )^{-\lambda} \frac{\mu'(s)}{\mu(s)} (|P^1(s)(z_1(s)-z_2(s))|+|P^2(s)(z_1(s)-z_2(s))|)\, ds \\
&\le  Dc\int_0^t \bigg (\frac{\mu(t)}{\mu(s)} \bigg )^{-\lambda} \frac{\mu'(s)}{\mu(s)} (|P^1(s)(z_1(s)-z_2(s))|_s+|P^2(s)(z_1(s)-z_2(s))|_s)\, ds \\
&\le Dc\int_0^t \bigg (\frac{\mu(t)}{\mu(s)} \bigg )^{-\lambda} \frac{\mu'(s)}{\mu(s)} |z_1(s)-z_2(s)|_s\, ds \\
&= \frac{Dc}{\lambda} \bigg (1-\frac{1}{\mu(t)^\lambda} \bigg ) \|z_1-z_2\|_\infty \\
&\le \frac{Dc}{\lambda}  \|z_1-z_2\|_\infty,
\end{split}
\]
and thus 
\begin{equation}\label{l1}
\begin{split}
& \sup_{t\ge 0}\bigg |\int_0^t T(t,s)P^1(s)(f(s, y(s)+(\bbI-P^3(s))z_1(s))-f(s, y(s)+(\bbI-P^3(s))z_2(s)))\, ds \bigg |_t \\
&\le \frac{Dc}{\lambda} \|z_1-z_2\|_\infty. 
\end{split}
\end{equation}
Similarly, 
\[
\begin{split}
& \bigg |\int_t^\infty T(t,s)P^2(s)(f(s, y(s)+(\bbI-P^3(s))z_1(s))-f(s, y(s)+(\bbI-P^3(s))z_2(s)))\, ds \bigg |_t \\
&\le D \int_t^\infty \bigg (\frac{\mu(s)}{\mu(t)} \bigg )^{-\lambda}\nu(s)^d |f(s, y(s)+(\bbI-P^3(s))z_1(s))-f(s, y(s)+(\bbI-P^3(s))z_2(s))| \, ds \\
&\le Dc\int_t^\infty \bigg (\frac{\mu(s)}{\mu(t)} \bigg )^{-\lambda} \frac{\mu'(s)}{\mu(s)} |(\bbI-P^3(s))z_1(s)-(\bbI-P^3(s))z_2(s)|\, ds \\
&\le  Dc\int_t^\infty \bigg (\frac{\mu(s)}{\mu(t)} \bigg )^{-\lambda} \frac{\mu'(s)}{\mu(s)} (|P^1(s)(z_1(s)-z_2(s))|+|P^2(s)(z_1(s)-z_2(s))|)\, ds \\
&\le  Dc\int_t^\infty \bigg (\frac{\mu(s)}{\mu(t)} \bigg )^{-\lambda} \frac{\mu'(s)}{\mu(s)} (|P^1(s)(z_1(s)-z_2(s))|_s+|P^2(s)(z_1(s)-z_2(s))|_s)\, ds \\
&\le Dc\int_t^\infty \bigg (\frac{\mu(s)}{\mu(t)} \bigg )^{-\lambda} \frac{\mu'(s)}{\mu(s)} |z_1(s)-z_2(s)|_s\, ds \\
&\le \frac{Dc}{\lambda}  \|z_1-z_2\|_\infty,
\end{split}
\]
and therefore
\begin{equation}\label{l2}
\begin{split}
& \sup_{t\ge 0}\bigg |\int_t^\infty T(t,s)P^2(s)(f(s, y(s)+(\bbI-P^3(s))z_1(s))-f(s, y(s)+(\bbI-P^3(s))z_2(s)))\, ds \bigg |_t \\
&\le \frac{Dc}{\lambda} \|z_1-z_2\|_\infty. 
\end{split}
\end{equation}
Finally, 
\[
\begin{split}
& |P^3(t)(f(t, y(t)+(\bbI-P^3(t))z_1(t))-f(t, y(t)+(\bbI-P^3(t))z_2(t)))|_t \\
&=\frac{\mu(t)}{\mu'(t)} |P^3(t)(f(t, y(t)+(\bbI-P^3(t))z_1(t))-f(t, y(t)+(\bbI-P^3(t))z_2(t)))| \\
&\le \frac{\mu(t)}{\mu'(t)}(2D+1)\nu(t)^d |f(t, y(t)+(\bbI-P^3(t))z_1(t))-f(t, y(t)+(\bbI-P^3(t))z_2(t))| \\
&\le c(2D+1)|(\bbI-P^3(t))z_1(t)-(\bbI-P^3(t))z_2(t)| \\
&\le c(2D+1)| z_1(t)-z_2(t)|_t,
\end{split}
\]
and hence
\begin{equation}\label{l3}
\begin{split}
&\sup_{t\ge 0}|P^3(t)(f(t, y(t)+(\bbI-P^3(t))z_1(t))-f(t, y(t)+(\bbI-P^3(t))z_2(t)))|_t \\
& \le c(2D+1)\|z_1-z_2\|_\infty. 
\end{split}
\end{equation}
From~\eqref{l1}, \eqref{l2} and~\eqref{l3}, we conclude that 
\begin{equation}\label{conc}
\|\mathcal Tz_1-\mathcal Tz_2\|_\infty \le q\|z_1-z_2\|_\infty.
\end{equation}
Set
\[
C:=\frac{\bar D}{1-q}, 
\]
and let
\[
\mathcal B:=\{ z\in \mathcal Y: \|z\|_\infty \le C\delta \}.
\]
We claim that $\mathcal T(\mathcal B)\subset \mathcal B$. Indeed, take an arbitrary $z\in \mathcal B$. It follows from~\eqref{T0} and~\eqref{conc} that 
\[
\begin{split}
\| \mathcal T z\|_\infty &\le \|\mathcal T0\|_\infty + \|\mathcal Tz-\mathcal T0\|_\infty \\
&\le \bar D\delta+q\|z\|_\infty \\
&\le \bar D\delta+qC\delta \\
&=C\delta, 
\end{split}
\]
and thus $\mathcal Tz\in \mathcal B$. From~\eqref{C} and~\eqref{conc}, we conclude that $\mathcal T\rvert_{\mathcal B} \colon \mathcal B\to \mathcal B$ is a contraction, and therefore it has a unique fixed point $z\in \mathcal B$. Hence, for each $t\ge 0$ we have that 
\begin{equation}\label{fp}
\begin{split}
z(t) &=-P^3(t)(A(t)y(t)+f(t, y(t)+(\bbI-P^3(t))z(t))-y'(t)) \\
&\phantom{=}+\int_0^t T(t,s)P^1(s)(A(s)y(s)+f(s, y(s)+(\bbI-P^3(s))z(s))-y'(s))\, ds \\
&\phantom{=}-\int_t^\infty T(t,s)P^2(s)(A(s)y(s)+f(s, y(s)+(\bbI-P^3(s))z(s))-y'(s))\, ds.
\end{split}
\end{equation}
Set 
\[
\begin{split}
\bar z(t) &=\int_0^t T(t,s)P^1(s)(A(s)y(s)+f(s, y(s)+(\bbI-P^3(s))z(s))-y'(s))\, ds \\
&\phantom{=}-\int_t^\infty T(t,s)P^2(s)(A(s)y(s)+f(s, y(s)+(\bbI-P^3(s))z(s))-y'(s))\, ds,
\end{split}
\]
for $t\ge 0$. Observe that $\bar z(t)=(\bbI-P^3(t))z(t)$, $t\ge 0$. Moreover, for $t\ge r\ge 0$ we have that 
\[
\begin{split}
&\bar z(t)-T(t,r)\bar z(r) \\
&=\int_0^t T(t,s)P^1(s)(A(s)y(s)+f(s, y(s)+(\bbI-P^3(s))z(s))-y'(s))\, ds  \\
&\phantom{=}-\int_0^rT(t,s)P^1(s)(A(s)y(s)+f(s, y(s)+(\bbI-P^3(s))z(s))-y'(s))\, ds  \\
&\phantom{=}-\int_t^\infty T(t,s)P^2(s)(A(s)y(s)+f(s, y(s)+(\bbI-P^3(s))z(s))-y'(s))\, ds\\
&\phantom{=}+\int_r^\infty T(t,s)P^2(s)(A(s)y(s)+f(s, y(s)+(\bbI-P^3(s))z(s))-y'(s))\, ds,
\end{split}
\]
and consequently 
\[
\begin{split}
&\bar z(t)-T(t,r)\bar z(r) \\
&=\int_r^t T(t,s)P^1(s)(A(s)y(s)+f(s, y(s)+(\bbI-P^3(s))z(s))-y'(s))\, ds  \\
&\phantom{=}+\int_r^t T(t,s)P^2(s)(A(s)y(s)+f(s, y(s)+(\bbI-P^3(s))z(s))-y'(s))\, ds.
\end{split}
\]
By differentiating, we obtain that
\[
\bar z'(t)=A(t)\bar z(t)+(P^1(t)+P^2(t))(A(t)y(t)+f(t, y(t)+\bar z(t))-y'(t)),
\]
for $t\ge 0$. This easily implies that 
\begin{equation}\label{sol}
\begin{split}
(y+\bar z)'(t) &=A(t)(y(t)+\bar z(t))+f(t, y(t)+\bar  z(t)) \\
&\phantom{=}-P^3(t)(A(t)y(t)+f(t, y(t)+\bar z(t))-y'(t)),
\end{split}
\end{equation}
for $t\ge 0$. We define $x\colon [0, +\infty)\to X$ by
\[
x(t)=y(t)+\bar z(t), \quad t\ge 0.
\]
Then, 
\[
P^3(t)(x(t)-y(t))=P^3(t)\bar z(t)=P^3(t)(\bbI-P^3(t))z(t)=0,
\]
for $t\ge 0$. Moreover, using~\eqref{ln} we have that 
\[
|\bar z(t)| =|(P^1(t)+P^2(t))z(t)|\le |P^1(t)z(t)|_t+|P^2(t)z(t)|_t \le |z(t)|_t \le \|z\|_\infty,
\]
for each $t\ge 0$. Thus, 
\[
\sup_{t\ge 0} |x(t)-y(t)| =\sup_{t\ge 0} |\bar z(t)| \le C\delta. 
\]
In addition, \eqref{sol} implies that $x'(t)-A(t)x(t)-f(t, x(t))\in \Ima P^3(t)$ for $t\ge 0$. It remains to establish~\eqref{EST}. We have that 
\begin{displaymath}
\begin{split}
& |x'(t)-A(t)x(t)-f(t, x(t))| \\
&=|-P^3(t)(A(t)y(t)+f(t, y(t)+\bar z(t))-y'(t))|\\
 &=|-P^3(t)(A(t)y(t)+f(t, y(t)+(\bbI-P^3(t))z(t))-y'(t))|\\
 &=|z(t)-\bar{z}(t)|\\
 &=|z(t)-(\bbI-P^3(t))z(t))|\\
 &=|P^3(t)z(t)|\\
 &\leq (2D+1)\nu (t)^d|z(t)|\\
 &\leq (2D+1)\nu (t)^d \max \bigg \{ 1, \frac{\mu'(t)}{\mu(t)} \bigg \} |z(t)|_t\\
 &\leq C\delta(2D+1)\nu (t)^d\max \bigg \{ 1, \frac{\mu'(t)}{\mu(t)} \bigg \}, 
\end{split}
\end{displaymath}
where we used that 
\[
\begin{split}
|z(t)| &\le |P^1(t)z(t)|+|P^2(t)z(t)|+|P^3(t)z(t)| \\
&\le |P^1(t)z(t)|_t+|P^2(t)z(t)|_t+\frac{\mu'(t)}{\mu(t)} |P^3(t)z(t)|_t \\
&\le \max \bigg \{ 1, \frac{\mu'(t)}{\mu(t)} \bigg \} (|P^1(t)z(t)|_t+|P^2(t)z(t)|_t+|P^3(t)z(t)|_t) \\
&=\max \bigg \{ 1, \frac{\mu'(t)}{\mu(t)} \bigg \} |z(t)|_t,
\end{split}
\]
for $t\ge 0$. Thus, \eqref{EST} holds  and the proof of the theorem is completed. 
\end{proof}

Obviously, Theorem~\ref{thm1} is most interesting in the particular case when $P^3(t)=0$ for $t\ge 0$. Indeed, in this case we can in the vicinity of each approximate solution of~\eqref{nde} construct an exact solution of~\eqref{nde}. More precisely, we have the following result.
\begin{corollary}\label{cor}
Assume that $P^3(t)=0$ for each $t\ge 0$ and that~\eqref{C} holds.  Then, there exists $C>0$ with the property that for any $\delta >0$ and an arbitrary differentiable map $y\colon [0, +\infty) \to X$ satisfying~\eqref{pseudo}, there exists a solution $x\colon [0, +\infty)\to X$ of~\eqref{nde} 
such that 
\begin{equation}\label{shad}
\sup_{t\ge 0} |x(t)-y(t)| \le C\delta. 
\end{equation}
\end{corollary}

\begin{proof}
The desired conclusion follows directly from Theorem~\ref{thm1}.
\end{proof}
We observe that Corollary~\ref{cor} is in particular applicable to a linear equation~\eqref{lde}.
\begin{corollary}\label{corx}
There exists $C>0$ with the property that for any $\delta >0$ and an arbitrary differentiable map $y\colon [0, +\infty) \to X$ satisfying
\[
|y'(t)-A(t)y(t)| \le \frac{\delta \mu'(t)}{\mu(t) \nu(t)^d} \quad \text{for $t\ge 0$,}
\]
there exists a solution $x\colon [0, +\infty)\to X$ of~\eqref{lde} such that~\eqref{shad} holds.
\end{corollary}

\begin{proof}
It sufficies to apply Corollary~\ref{cor} in the case when $f(t,x)=0$ for $t\ge 0$ and $x\in X$. Observe that in this case~\eqref{non} holds with $c=0$ and thus~\eqref{C} is obviously satisfied.
\end{proof}

\section{Examples}\label{sec: examples}
In this section we discuss a variety of settings to which our results are applicable. In addition, we discuss the relationship between Theorem~\ref{thm1} and the results already available in the literature.
\subsection{The case of exponential dichotomies}\label{ED}
This situation corresponds to the case when $P^3(t)=0$, $\mu(t)=e^t$ and $\nu(t)=1$, $t\ge 0$. Observe that in this setting, \eqref{non} reduces to
\[
|f(t,x)-f(t,y)| \le c|x-y|, \quad \text{for $t\ge 0$ and $x, y\in X$.}
\]
Moreover, \eqref{pseudo} is equivalent to
\[
|y'(t)-A(t)y(t)-f(t, y(t))| \le \delta \quad \text{for $t\ge 0$,}
\]
Hence, Corollary~\ref{cor} is a analogous to~\cite[Theorem 5]{BD1}  except that in the present paper, we consider dynamics on $[0, +\infty)$ instead of $\mathbb R$. We stress that the proof of~\cite[Theorem 5]{BD1} is done by passing from the continuous time dynamics to the associated discrete time dynamics, while in the present paper we developed a direct approach.

\subsection{The case of partial exponential dichotomies}
In this setting we continue to assume that  $\mu(t)=e^t$ and $\nu(t)=1$, $t\ge 0$. However, we don't require that $P^3(t)=0$. In this context, Theorem~\ref{thm1} represents a continuous time version of~\cite[Theorem 2.]{BD2}. We stress that in~\cite[Section 5.]{BD2}, a different continuous time version of~\cite[Theorem 2.]{BD2} (see~\cite[Theorem 4.]{BD2}) was proposed and again obtained by suitable discretization of time. However, Theorem~\ref{thm1} presents a more natural continuous time version of~\cite[Theorem 2.]{BD2} (when compared to~\cite[Theorem 4.]{BD2}), since $x$ given by Theorem~\ref{thm1} is a solution of~\eqref{nde} up to the movement in the central direction (i.e. along images of $P^3(t)$). We emphasize that this is precisely what corresponds to~\cite[Theorem 2.]{BD2} in the case of discrete time. 

We discuss a particular example to which Theorem~\ref{thm1} is applicable in this setting.
\begin{example}
Take $X=\mathbb R^3$ and set
\[
A(t)=\begin{pmatrix}
-(1+2t) & 0 & 0 \\
0 & 1+2t & 0 \\
0 & 0 & 0
\end{pmatrix}, \quad t\ge 0.
\]
Then, it is easy to show that~\eqref{d1} and~\eqref{d2} holds with $D=\lambda =1$ and
\[
P^1(x_1, x_2,x_3)=(x_1, 0, 0) \quad \text{and} \quad P^2(x_1,x_2, x_3)=(0, x_2, 0), 
\]
$(x_1,x_2,x_3)\in \mathbb R^3$. Take $c\ge 0$ and set
\[
f(t,x)=c\sin |x|, \quad (t,x)\in [0, +\infty)\times \mathbb R^3.
\]
Obviously, \eqref{non} holds. Hence, Theorem~\ref{thm1} is applicable whenever $c$ is sufficiently small.
\end{example}

\subsection{The case of tempered exponential dichotomies}
Here, we take $P^3(t)=0$ and $\mu(t)=e^t$, $t\ge 0$. In addition, we assume that $\nu$ is tempered, i.e. that 
\begin{equation}\label{temp}
\lim_{t\to +\infty} \frac 1 t \ln \nu (t)=0.
\end{equation}
Let us discuss a particular example.
\begin{example}
Take $X=\mathbb R$ and set
\[
A(t)=-1+\frac{\cos t}{\sqrt{1+t}}-\sqrt{1+t}\sin t, \quad t\ge 0.
\]
One can easily show that
\[
T(t,s)=e^{-(t-s)+\sqrt{1+t} \cos t-\sqrt{1+s}\cos s}, \quad t,s\ge 0.
\]
Obviously, 
\[
T(t,s)\le e^{-(t-s)+\sqrt{1+t} +\sqrt{1+s}}, \quad t\ge s\ge 0.
\]
Since 
\[
\sqrt{1+t}-\sqrt{1+s} \le \frac 1 2 (t-s) \quad t\ge s\ge 0,
\]
we have that 
\[
T(t,s) \le e^{-\frac 1 2(t-s)+2\sqrt{1+s}},  \quad t\ge s\ge 0.
\]
Therefore, \eqref{d1} and~\eqref{d2} hold with $\lambda=\frac 1 2$, $D=1$, $P^1(t)=\bbI$, $P^2(t)=0$, $d=2$ and $\nu(t)=e^{\sqrt{1+t}}$ (which obviously satisfies~\eqref{temp}). Take $c\ge 0$ and set
\[
f(t,x)=\frac{c}{\nu(t)^d} \sin x, \quad (t,x)\in [0, +\infty)\times \mathbb R.
\]
We see that~\eqref{non} holds. Again, Theorem~\ref{thm1} is applicable whenever $c$ is sufficiently small. 

In this context, Corollary~\ref{cor} is similar to the results established in~\cite{BD3}, although in that paper we dealt with random dynamics. 
\end{example}
\subsection{The case of nonuniform polynomial dichotomies}\label{PD}

Here we take $\mu(t)=\nu(t)=t+1$ and $P^3(t)=0$ for $t\ge 0$ (see~\cite{BS}). Hence, \eqref{non} reads as
\begin{equation}\label{pnon}
|f(t,x)-f(t,y)| \le \frac{c}{(1+t)^{d+1}}|x-y|, \quad \text{for $t\ge 0$ and $x, y\in X$,}
\end{equation}
while~\eqref{pseudo} means that 
\[
|y'(t)-A(t)y(t)-f(t, y(t))| \le \frac{\delta}{(1+t)^{d+1}} \quad \text{for $t\ge 0$.}
\]
To the best of our knowledge, in the present context, Corollary~\ref{cor} represents a first result discussing the shadowing property for~\eqref{nde} when~\eqref{lde} admits a polynomial dichotomy.

For the sake of completeness we discuss a concrete example in this setting to which Corollary~\ref{cor} is applicable. 

\begin{example}
Take $X=\mathbb R^2$,  $d>0$ and $a<0\le b$. Furthermore, for $t\ge 0$ set
\[
A(t)=\begin{pmatrix}
a(t) & 0 \\
0 & b(t)
\end{pmatrix},
\]
where
\[
a(t)=\frac{a}{t+1}+\frac{d}{2(t+1)}(\cos t-1)-\frac{d}{2} \ln (1+t)\sin t
\]
and 
\[
b(t)=\frac{b}{t+1}-\frac{d}{2(t+1)}(\cos t-1)+\frac{d}{2} \ln (1+t)\sin t .
\]
It is proved in~\cite[Example 2.1.]{BS} that~\eqref{d1} and~\eqref{d2} hold with $D=1$, $\lambda =\min \{-a, b\}>0$ and 
\[
P^1(x,y)=(x, 0) \quad \text{and} \quad P^2(x,y)=(0, y), \quad (x, y)\in \mathbb R^2.
\] 
Take $c\ge 0$ and set
\[
f(t,x)=\frac{c}{(1+t)^{d+1}}\sin |x|, \quad (t,x)\in [0, +\infty)\times \mathbb R^2.
\]
Clearly, \eqref{pnon} holds and thus  Corollary~\ref{cor} is applicable for $c$ sufficiently small.
\end{example}

\subsection{The case of $(\mu, \nu)$-dichotomies}\label{sec: mu nu dichotomies} 
The notion of $(\mu, \nu)$-dichotomy is introduced by Bento and Silva~\cite{BS1} and  corresponds to our general setting with $P^3(t)=0$. We observe that this setting includes the ones discussed in Subsections~\ref{ED} and~\ref{PD} as particular cases. 
We refer to~\cite{BS} for particular examples of this behaviour (which don't belong to the  classes  of examples discussed in Subsections~\ref{ED} and~\ref{PD}). Again, Corollary~\ref{cor} represents a first shadowing result for~\eqref{nde} in the case when~\eqref{lde} admits a  $(\mu, \nu)$-dichotomy.

\section{Equations of higher order}\label{sec: higher order}
The purpose of this section is to indicate that our results apply for differential equations of higher order.

  We take continuous functions $A, B \colon [0, +\infty) \to \mathcal B(X)$ and a continuous function $f\colon [0,+\infty)\times X\times X \to X$. We consider a nonlinear differential equation of second order given by
\begin{equation}\label{hon}
x''=A(t)x'+B(t)x+f(t, x', x) \quad t\ge 0,
\end{equation}
and the associated linear equation
\begin{equation}\label{huh}
x''=A(t)x'+B(t)x, \quad t\ge 0.
\end{equation}

Let $X^2=X\times X$. Then, $X^2$ is a Banach space with respect to the norm $|(x_1, x_2)|'=|x_1|+|x_2|$, $(x_1, x_2)\in X^2$. For $t\ge 0$, we define $C(t)\colon X^2\to X^2$ by
\[
C(t)(x_1,x_2)=(A(t)x_1+B(t)x_2, x_1), \quad (x_1, x_2)\in X^2.
\]
We consider a linear differential equation on $X^2$ given by
\begin{equation}\label{ldex}
z'=C(t)z.
\end{equation}
\begin{corollary}\label{cor2}
Assume that~\eqref{ldex} admits a $(\mu, \nu)$-dichotomy (see Section \ref{sec: mu nu dichotomies}). Furthermore, suppose that there exists $c\ge 0$ such that
\begin{equation}\label{NE}
|f(t,x_1,x_2)-f(t,x_1', x_2')| \le \frac{c\mu'(t)}{\mu(t)\nu(t)^d} (|x_1-x_1'|+|x_2-x_2'|), 
\end{equation}
for $t\ge 0$ and $(x_1,x_2), (x_1', x_2')\in X^2$. Then, for every $c$ sufficiently small, there exists $C>0$ with the property that for any $\delta >0$ and an arbitrary twice differentiable map $y\colon [0, +\infty)\to X$ satisfying
\begin{equation}\label{P}
\sup_{t\ge 0} |y''(t)-A(t)y'(t)-B(t)y(t)-f(t, y'(t), y(t))| \le  \frac{\delta \mu'(t)}{\mu(t)\nu(t)^d}, 
\end{equation}
there exists a solution $x\colon [0, +\infty) \to X$ of~\eqref{hon} such that 
\begin{equation}\label{aprox}
\sup_{t\ge 0} |x(t)-y(t)| \le C\delta.
\end{equation}
\end{corollary}

\begin{proof}
We define $w\colon [0, +\infty)\to X^2$ by
\[
w(t)=(y'(t), y(t)), \quad t\ge 0.
\]
Furthermore, we define $g\colon [0, +\infty)\times X^2\to X^2$ by
\[
g(t,x_1, x_2)=(f(t,x_1, x_2), 0), \quad (x_1, x_2)\in X^2.
\]
By~\eqref{NE}, we have that 
\begin{equation}\label{NE1}
|g(t,x_1,x_2)-g(t,x_1', x_2')|' \le \frac{c\mu'(t)}{\mu(t)\nu(t)^d} |(x_1, x_2)-(x_1', x_2')|',
\end{equation}
for $t\ge 0$ and $(x_1,x_2), (x_1', x_2')\in X^2$.  
Moreover, it follows from~\eqref{P} that 
\begin{equation}\label{P1}
\sup_{t\ge 0} |w'(t)-C(t)w(t)-g(t,w(t))|' \le \frac{\delta \mu'(t)}{\mu(t)\nu(t)^d}
\end{equation}
Then, for $c$ sufficiently small, we can apply Corollary~\ref{cor} and conclude that there exists a differentiable map  $z=(z_1, z_2)\colon [0, +\infty)\to X^2$ such that 
\begin{equation}\label{z}
z'(t)=C(t)z(t)+g(t, z(t))\quad t\ge 0,
\end{equation}
and
\begin{equation}\label{ij}
\sup_{t\ge 0} |z(t)-w(t)|'\le C\delta,
\end{equation}
where $C>0$ comes from Corollary~\ref{cor}. From~\eqref{z}, it follows easily that $x:=z_2$ is a solution of~\eqref{hon}. Moreover, \eqref{ij} implies that~\eqref{aprox} holds. The proof is completed. 
\end{proof}

\begin{remark}
In the particular case when $X=\mathbb C$, $\mu(t)=e^t$, $\nu(t)=1$, $t\ge 0$ and with $A$ and $B$ being periodic, the result similar  to Corollary~\ref{cor2} was obtained in~\cite[Theorem 2.12.]{BLR}. Indeed, the only difference is that in~\cite{BLR}, the authors consider dynamics on $\mathbb R$ while we consider dynamics on $[0, +\infty)$.
\end{remark}

\begin{remark}
In Corollary~\ref{cor2}, we have for the sake of simplicity focused on equations of order $2$. However, it is clear that an analogous result can be formulated for equations of any order.
\end{remark}

\section{The case of discrete time}\label{sec: discrete}
In this section we present a discrete time version of Theorem \ref{thm1}. We start by introducing some useful notations. Let $(X, |\cdot |)$ and $(\mathcal B(X),\|\cdot\|)$ be as in Section \ref{sec: prelim}. We denote by $\bbN$ the set of all natural numbers, including $0$, while $\bbN^\ast$ denotes the set of all natural numbers excluding $0$. Given a sequence $(A_n)_{n\in \mathbb{N}}$ of bounded linear operators in $\mathcal B(X)$ and $m, n\in \mathbb{N}$, let us consider 
\[
\cA (m, n)=\begin{cases}
A_{m-1}\cdots A_n & \text{for $m>n$;}\\
\bbI  &\text{for $m=n$.} \\
\end{cases}
\]

Let $ \mu =(\mu_n)_{n\in \mathbb{N}}$ be a strictly  increasing sequence such that 
\[\mu_0\geq 1 \quad  \text{and} \quad \lim_{n\to +\infty}\mu_n=+\infty.\]
Moreover, let $ \nu =(\nu_n)_{n\in \mathbb{N}}$ be an arbitrary sequence with $\nu _n \geq 1$ for every $n\in \bbN$.

All along this section we are going to assume that the following conditions are satisfied:
\begin{enumerate}
\item  there exist three families of projections $P^i_n$, $n\in \mathbb{N}$, $i\in \{1, 2, 3\}$ such that 
\[
P^1_n+P^2_n+P^3_n=\bbI, \quad n\in \bbN;
\]
\item for $i, j\in \{1, 2,3 \}$, $i\neq j$ and $n\in \bbN$, 
\[
P^i_nP^j_n=0;
\]
\item for every $n\in \bbN$ and $i\in \{1, 2, 3\}$, we have that 
\[
A_nP^i_n=P^i_{n+1}A_n;
\]
\item $A_n\rvert_{\Ima P^2_n} \colon \Ima P^2_n \to \Ima P^2_{n+1}$ is an invertible operator for each $n\in \bbN$;
\item there exist $D, \lambda >0$ and $d\ge 0$  such that 
\begin{equation}\label{eq: Es discrete}
\|\cA(m,n)P_n^1\|\leq D\left(\frac{\mu_m}{\mu_n}\right)^{-\lambda}\nu_n^d \quad \text{ for } m\geq n
\end{equation}
and
\begin{equation}\label{eq: Eu discrete}
\|\cA(m,n)P_n^2\|\leq D\left(\frac{\mu_n}{\mu_m}\right)^{-\lambda}\nu_n^d \quad \text{ for } m\leq n
\end{equation}
where 
\[
\cA(m, n):=\big{(}\cA(n, m)\rvert_{\Ima P^2_m} \big{)}^{-1} \colon \Ima P^2_n \to \Ima P^2_m,
\]
for $m\le n$. 

\end{enumerate}

For $n\in \bbN$, set 
\[
E^s_n=\Ima P^1_n, \ E^u_n=\Ima P^2_n \ \text{and} \ E^c_n=\Ima P^3_n.
\]
Furthermore, we set
\[
|x|_n:= \sup_{m\ge n} \bigg (|\cA(m,n)x|  \bigg (\frac{\mu_m}{\mu _n} \bigg )^{\lambda}   \bigg) \quad \text{for $n\in \bbN$ and $x\in E^s_n$,}
\]
\[
|x|_n:= \sup_{m\le n} \bigg (|\cA(m,n)x| \bigg (\frac{\mu_n}{\mu_m} \bigg )^{\lambda} \bigg) \quad \text{for $n\in \bbN$ and $x\in E^u_n$,}
\]
and 
\[
|x|_n:= \left(\frac{\mu^{\lambda}_{n+1}}{\mu^{\lambda}_{n+1}-\mu_n^\lambda }\right)|x|, \quad \text{for $n\in \bbN$ and $x\in E^c_n$.}
\]

Finally, for an arbitrary $x\in X$ let 
\[
|x|_n:= |P^1_nx|_n+|P^2_nx|_n+|P^3_nx|_n.
\]
It is easy to see that $|\cdot |_n$ is a norm on $X$ for every $n\in \bbN$. Moreover,  by repeating the arguments  in the proof of Lemma \ref{lemma: comparation norms}, one can establish the following result.
\begin{lemma} \label{lemma: comparation norms discrete}
The following properties hold:
\begin{itemize}
\item for $n\in \bbN$ and $x\in E^s_n\cup E^u_n$, 
\begin{equation}\label{ln disc}
|x|\le |x|_n ;
\end{equation}
\item for $n\ge m$ and $x\in X$, 
\begin{equation}\label{ln1 disc}
|\cA(n,m)P^1_mx|_n \le D \bigg (\frac{\mu_n}{\mu _m} \bigg )^{-\lambda} \nu_m^d|x|;
\end{equation}
\item for $n\le m$ and $x\in X$, 
\begin{equation}\label{ln2 disc}
|\cA(n,m)P^2_mx|_n \le D \bigg (\frac{\mu_m}{\mu_n} \bigg )^{-\lambda}\nu_m^d|x|.
\end{equation}
\end{itemize}
\end{lemma}

\subsection{Main result for the case of discrete time} Let $f_n:X\to X$ be a sequence of maps and suppose there exist $c\ge 0$ such that  
\begin{equation}\label{non disc} 
|f_n(x)-f_n(y)| \leq c\left(\frac{\mu^{\lambda}_{n+1}-\mu_n^\lambda}{\mu_{n+1}^\lambda \nu_n^d}\right)|x-y|, \quad \text{for $n\in \bbN$ and $x, y\in X$.}
\end{equation}

The next result is a version of Theorem \ref{thm1} in the present setting.

\begin{theorem}\label{thm2}
Suppose that 
\begin{equation}\label{C disc}
q:=c(4D+1) <1.
\end{equation}
Then, there exists $C>0$ with the property that for any $\delta >0$ and an arbitrary sequence $\mathbf{y}=(y_n)_{n\in \bbN}\subset X$ satisfying
\begin{equation}\label{pseudo disc}
|y_{n+1}-A_ny_n-f_n(y_n)| \leq \delta \left(\frac{\mu^{\lambda}_{n+1}-\mu_n^\lambda}{\mu_{n+1}^\lambda \nu_n^d}\right) \quad \text{for every $n\in \bbN$,}
\end{equation}
there exists a sequence $\mathbf{x}=(x_n)_{n\in \bbN}\subset X$ such that:
\begin{itemize}
\item $P^3_n(x_n-y_n)=0$ for every $n\in \bbN$;
\item $\sup_{n\in \bbN} |x_n-y_n| \le C\delta$;
\item $x_{n+1}-A_nx_n-f_n(x_n)\in \Ima P^3_n$ for every $n\in \bbN$;
\item for every $n\in \mathbb N$,
\begin{equation}\label{EST2}
|x_{n+1}-A_n x_n-f_n(x_n)|  \le  C\delta (2D+1)\nu_n^d.
\end{equation}
\end{itemize}
\end{theorem}
\begin{remark}
As its continuous time counterpart, Theorem \ref{thm2} may be seen as a version of the shadowing property for the nonlinear dynamics
\begin{displaymath}
x_{n+1}=A_nx_n+f_n(x_n),\quad n\in \bbN.
\end{displaymath}
\end{remark}

\begin{remark}
We emphasize that we don't require operators $A_n$ to be invertible. Indeed, it is only assumed that $A_n\rvert_{E_n^u} \colon  E_n^u \to E_{n+1}^u$ is invertible for each $n\in \mathbb N$.
\end{remark}
\begin{proof}[Proof of  Theorem~\ref{thm2}]
We follow closely the proof of Theorem~\ref{thm1}.

Let
\[
\mathcal Y:=\bigg \{ \mathbf{z}=(z_n)_{n\in \mathbb N} \subset X:  \ \|\mathbf{z}\|_\infty:=\sup_{n\in \bbN} |z_n|_n <+\infty \bigg \}.
\]
Then, $(\mathcal Y, \|\cdot \|_\infty)$ is a Banach space. For $\mathbf z\in \mathcal Y$, we set $(\mathcal T \mathbf{z})_0=0$ and 
\[
\begin{split}
(\mathcal T \mathbf{z})_n &=-P^3_{n-1}(A_{n-1} y_{n-1}+f_{n-1}(y_{n-1}+(\bbI-P^3_{n-1})z_{n-1})-y_{n}) \\
&\phantom{=}+\sum _{m=0}^{n-1} \cA(n,m) P^1_m(A_my_m+f_m(y_m+(\bbI-P^3_m)z_m)-y_{m+1}) \\
&\phantom{=}-\sum_{m=n}^{+\infty} \cA(n,m)P^2_m(A_my_m+f_m(y_m+(\bbI-P^3_m)z_m)-y_{m+1})
\end{split}
\]
for all $n\in \bbN^\ast$.

We start observing that it follows from~\eqref{ln1 disc} and~\eqref{pseudo disc} that 
\[
\begin{split}
 & \bigg |\sum _{m=0}^{n-1} \cA(n,m) P^1_m(A_my_m+f_m(y_m)-y_{m+1}) \bigg |_n \\
&\le \sum _{m=0}^{n-1} |\cA(n,m) P^1_m(A_my_m+f_m(y_m)-y_{m+1})|_n
\\
&\le D\sum _{m=0}^{n-1} \bigg (\frac{\mu_n}{\mu_m} \bigg )^{-\lambda} \nu_m^d|A_my_m+f_m(y_m)-y_{m+1}| \\
&\le D\delta \sum _{m=0}^{n-1} \bigg (\frac{\mu_n}{\mu_m} \bigg )^{-\lambda} \left(\frac{\mu^{\lambda}_{m+1}-\mu_m^\lambda}{\mu_{m+1}^\lambda}\right)  \\
&=D\delta \frac{1}{\mu_n^\lambda}\sum _{m=0}^{n-1} \frac{\mu^{\lambda}_{m+1}-\mu_m^\lambda}{\mu_m^{-\lambda}\mu_{m+1}^\lambda}\\
&\leq D\delta \frac{1}{\mu_n^\lambda}\sum _{m=0}^{n-1} \left(\mu^{\lambda}_{m+1}-\mu_m^\lambda\right) \\
&\leq D\delta \bigg (1-\frac{1}{ \mu_n^\lambda} \bigg ),
\end{split}
\]
for $n\in \mathbb N^*$ and thus
\begin{equation}\label{01 disc}
\sup_{n\in \bbN^\ast} \bigg |\sum _{m=0}^{n-1} \cA(n,m) P^1_m(A_my_m+f_m(y_m)-y_{m+1}) \bigg |_n   \le D\delta.
\end{equation}

Similarly, using~\eqref{ln2 disc} and~\eqref{pseudo disc} we have that 
\[
\begin{split}
& \bigg | \sum_{m=n}^{+\infty} \cA(n,m)P^2_m(A_my_m+f_m(y_m)-y_{m+1}) \bigg |_n \\
&\le D \sum_{m=n}^{+\infty} \bigg (\frac{\mu_m}{\mu_n} \bigg )^{-\lambda} \nu_m^d |A_my_m+f_m(y_m)-y_{m+1} | \\
&\le D\delta \sum_{m=n}^{+\infty} \bigg (\frac{\mu_m}{\mu_n} \bigg )^{-\lambda} \left(\frac{\mu^{\lambda}_{m+1}-\mu_m^\lambda}{\mu_{m+1}^\lambda}\right)  \\
&= D\delta \mu_n^\lambda \sum_{m=n}^{+\infty}  \left(\frac{\mu^{\lambda}_{m+1}-\mu_m^\lambda}{\mu_m^\lambda \mu_{m+1}^\lambda}\right)  \\
&= D\delta \mu_n^\lambda \sum_{m=n}^{+\infty}  \left(\frac{1}{\mu_m^\lambda}-\frac{1}{\mu_{m+1}^\lambda}\right)  \\
&=D\delta,
\end{split}
\]
and consequently 
\begin{equation}\label{02 disc}
\sup_{n\in \bbN}  \bigg | \sum_{m=n}^{+\infty} \cA(n,m)P^2_m(A_my_m+f_m(y_m)-y_{m+1}) \bigg |_n \le D\delta.
\end{equation}
Moreover, we have 
\begin{equation}\label{aux disc}
\begin{split}
& |P^3_n(A_ny_n+f_n(y_n)-y_{n+1}) |_n \\
&=\left(\frac{\mu^{\lambda}_{n+1}}{\mu_{n+1}^\lambda-\mu_n^\lambda }\right) |P^3_n(A_ny_n+f_n(y_n)-y_{n+1}) | \\
&\le \left(\frac{\mu^{\lambda}_{n+1}}{\mu_{n+1}^\lambda-\mu_n^\lambda }\right) (2D+1)\nu_n^d |A_ny_n+f_n(y_n)-y_{n+1}|,
\end{split}
\end{equation}
and therefore 
\begin{equation}\label{03 disc}
\sup_{n\in \bbN} |P^3_n(A_ny_n+f_n(y_n)-y_{n+1}) |_n \le \delta (2D+1).
\end{equation}
Observe that in~\eqref{aux disc}, we have used (see~\eqref{eq: Es discrete} and~\eqref{eq: Eu discrete}) that 
\[
|P^3_nx| \le |x|+|P^1_nx|+|P^2_nx|  \le |x|+2D\nu_n^d |x| \le (2D+1)\nu_n^d |x|,
\]
for every $n\in \bbN$ and $x\in X$. Combining~\eqref{01 disc}, \eqref{02 disc} and~\eqref{03 disc}, we conclude  that
\begin{equation}\label{T0 disc}
\|\mathcal T \mathbf{0}\|_\infty \le \bar D\delta,
\end{equation}
where
\[
\bar D:=4D+1>0.
\]

Now, let $\mathbf{z}=(z_n)_{n\in \mathbb N}, \mathbf{w}=(w_n)_{n\in \mathbb N}\in \mathcal Y$ be arbitrary. By~\eqref{non disc}, \eqref{ln disc} and~\eqref{ln1 disc}, we have that 
\[
\begin{split}
& \bigg |\sum _{m=0}^{n-1} \cA(n,m) P^1_m(f_m(y_m+(\bbI-P^3_m)z_m)-f_m(y_m+(\bbI-P^3_m)w_m) )\bigg |_n \\
&\le D \sum _{m=0}^{n-1} \bigg (\frac{\mu_n}{\mu_m} \bigg )^{-\lambda}\nu_m^d |f_m(y_m+(\bbI-P^3_m)z_m)-f_m(y_m+(\bbI-P^3_m)w_m)|  \\
&\le Dc\sum _{m=0}^{n-1} \bigg (\frac{\mu_n}{\mu_m} \bigg )^{-\lambda} \left(\frac{\mu^{\lambda}_{m+1}-\mu_m^\lambda}{\mu_{m+1}^\lambda}\right)  |(\bbI-P^3_m)z_m-(\bbI-P^3_m)w_m| \\
&\le  Dc\sum _{m=0}^{n-1} \bigg (\frac{\mu_n}{\mu_m} \bigg )^{-\lambda} \left(\frac{\mu^{\lambda}_{m+1}-\mu_m^\lambda}{\mu_{m+1}^\lambda}\right) (|P^1_m(z_m-w_m)|+|P^2_m(z_m-w_m)|) \\
&\le  Dc\sum _{m=0}^{n-1} \bigg (\frac{\mu_n}{\mu_m} \bigg )^{-\lambda} \left(\frac{\mu^{\lambda}_{m+1}-\mu_m^\lambda}{\mu_{m+1}^\lambda}\right)  (|P^1_m(z_m-w_m)|_m+|P^2_m(z_m-z_m)|_m) \\
&\le Dc \frac{1}{\mu_n^\lambda}\sum _{m=0}^{n-1} \left(\frac{\mu^{\lambda}_{m+1}-\mu_m^\lambda}{\mu_m^{-\lambda}\mu_{m+1}^\lambda}\right) |z_m-w_m|_m \\
&\le Dc \frac{1}{\mu_n^\lambda}\sum _{m=0}^{n-1} \left(\mu^{\lambda}_{m+1}-\mu_m^\lambda\right) \|\mathbf{z}-\mathbf{w}\|_\infty \\
&\le Dc  \|\mathbf{z}-\mathbf{w}\|_\infty,
\end{split}
\]
and thus 
\begin{equation}\label{l1 disc}
\begin{split}
& \sup_{n\in \bbN ^\ast}\bigg |\sum _{m=0}^{n-1} \cA(n,m) P^1_m(f_m(y_m+(\bbI-P^3_m)z_m)-f_m(y_m+(\bbI-P^3_m)w_m)\bigg |_n \\
&\le Dc \|\mathbf{z}-\mathbf{w}\|_\infty. 
\end{split}
\end{equation}

Similarly, 
\[
\begin{split}
& \bigg |\sum_{m=n}^{+\infty} \cA(n,m)P^2_m(f_m(y_m+(\bbI-P^3_m)z_m)-f_m(y_m+(\bbI-P^3_m)w_m)) \bigg |_n \\
&\le D \sum_{m=n}^{+\infty} \bigg (\frac{\mu_m}{\mu_n} \bigg )^{-\lambda}\nu_m^d |f_m(y_m+(\bbI-P^3_m)z_m)-f_m(y_m+(\bbI-P^3_m)w_m)| \\
&\le Dc\sum_{m=n}^{+\infty} \bigg (\frac{\mu_m}{\mu_n} \bigg )^{-\lambda} \left(\frac{\mu^{\lambda}_{m+1}-\mu_m^\lambda}{\mu_{m+1}^\lambda}\right) |(\bbI-P^3_m)z_m-(\bbI-P^3_m)w_m| \\
&\le  Dc\sum_{m=n}^{+\infty} \bigg (\frac{\mu_m}{\mu_n} \bigg )^{-\lambda} \left(\frac{\mu^{\lambda}_{m+1}-\mu_m^\lambda}{\mu_{m+1}^\lambda}\right) (|P^1_m(z_m-w_m)|+|P^2_m(z_m-w_m)|) \\
&\le  Dc\sum_{m=n}^{+\infty} \bigg (\frac{\mu_m}{\mu_n} \bigg )^{-\lambda} \left(\frac{\mu^{\lambda}_{m+1}-\mu_m^\lambda}{\mu_{m+1}^\lambda}\right) (|P^1_m(z_m-w_m)|_m+|P^2_m(z_m-w_m)|_m) \\
&\le Dc \mu_n^\lambda \sum_{m=n}^{+\infty} \left(\frac{\mu^{\lambda}_{m+1}-\mu_m^\lambda}{\mu_m^{\lambda}\mu_{m+1}^\lambda}\right) |z_m-w_m|_m \\
&\le Dc \mu_n^\lambda \sum_{m=n}^{+\infty} \left(\frac{1}{\mu_m^{\lambda}}-\frac{1}{\mu_{m+1}^\lambda}\right) \|\mathbf{z}-\mathbf{w}\|_\infty \\
&\le Dc  \|\mathbf{z}-\mathbf{w}\|_\infty,
\end{split}
\]
and therefore
\begin{equation}\label{l2 disc}
\begin{split}
& \sup_{n\in \bbN}\bigg |\sum_{m=n}^{+\infty} \cA(n,m)P^2_m(f_m(y_m+(\bbI-P^3_m)z_m)-f_m(y_m+(\bbI-P^3_m)w_m)) \bigg |_n \\
&\le Dc \|\mathbf{z}-\mathbf{w}\|_\infty. 
\end{split}
\end{equation}
Finally, 
\[
\begin{split}
& |P^3_n(f_n(y_n+(\bbI-P^3_n)z_n)-f_n(y_n+(\bbI-P^3_n)w_n))|_n \\
&=\left(\frac{\mu_{n+1}^\lambda}{\mu^\lambda _{n+1}-\mu_n^\lambda}\right) |P^3_n(f_n(y_n+(\bbI-P^3_n)z_n)-f_n(y_n+(\bbI-P^3_n)w_n))| \\
&\le \left(\frac{\mu_{n+1}^\lambda}{\mu^\lambda _{n+1}-\mu_n^\lambda}\right)(2D+1)\nu_n^d |f_n(y_n+(\bbI-P^3_n)z_n)-f_n(y_n+(\bbI-P^3_n)w_n)| \\
&\le c(2D+1)|(\bbI-P^3_n)z_n-(\bbI-P^3_n)w_n| \\
&\le c(2D+1)| z_n-w_n|_n,
\end{split}
\]
and hence
\begin{equation}\label{l3 disc}
\begin{split}
&\sup_{n\in \bbN}|P^3_n(f_n(y_n+(\bbI-P^3_n)z_n)-f_n(y_n+(\bbI-P^3_n)w_n))|_n \\
& \le c(2D+1)\|\mathbf{z}-\mathbf{w}\|_\infty. 
\end{split}
\end{equation}
Thus, combining~\eqref{l1 disc}, \eqref{l2 disc} and~\eqref{l3 disc}, we conclude that 
\begin{equation}\label{conc disc}
\|\mathcal T \mathbf{z}-\mathcal T \mathbf{w} \|_\infty \le q\|\mathbf{z} - \mathbf{w} \|_\infty.
\end{equation}
Set
\[
C:=\frac{\bar D}{1-q}, 
\]
and consider
\[
\mathcal B:=\{ \mathbf{z} \in \mathcal Y: \|\mathbf{z} \|_\infty \le C\delta \}.
\]
By arguing as in the proof of Theorem~\ref{thm1}, it follows easily from~\eqref{T0 disc} and~\eqref{conc disc} that 
 $\mathcal T\mathbf{z}\in \mathcal B$. From~\eqref{C disc} and~\eqref{conc disc}, we conclude that $\mathcal T\rvert_{\mathcal B} \colon \mathcal B\to \mathcal B$ is a contraction, and therefore it has a unique fixed point $\mathbf{z}=(z_n)_{n\in \bbN}\in \mathcal B$.  Hence, for each $n\in \bbN^\ast$ we have that 
\begin{equation}\label{fp disc}
\begin{split}
z_n &=-P^3_{n-1}(A_{n-1}y_{n-1}+f_{n-1}(y_{n-1}+(\bbI-P^3_{n-1})z_{n-1})-y_{n}) \\
&\phantom{=}+\sum_{m=0}^{n-1} \cA(n,m) P^1_m(A_my_m+f_m(y_m+(\bbI-P^3_m)z_m)-y_{m+1}) \\
&\phantom{=}-\sum_{m=n}^{+\infty} \cA(n,m) P^2_m(A_my_m+f_m(y_m+(\bbI-P^3_m)z_m)-y_{m+1}).
\end{split}
\end{equation}

Now, let us consider $\bar{\mathbf{z}}=(\bar{z}_n)_{n\in \bbN}$ given by $\bar{z}_0=0$ and
\[
\begin{split}
\bar{z}_n &=\sum_{m=0}^{n-1} \cA(n,m) P^1_m(A_my_m+f_m(y_m+(\bbI-P^3_m)z_m)-y_{m+1}) \\
&\phantom{=}-\sum_{m=n}^{+\infty} \cA(n,m) P^2_m(A_my_m+f_m(y_m+(\bbI-P^3_m)z_m)-y_{m+1})
\end{split}
\]
for every $n\in \bbN^\ast$. Observe that $\bar{z}_n=(\bbI-P^3_n)z_n$, $n\in \bbN$. Moreover, for every $n> k> 0$ we have that 
\[
\begin{split}
\bar{z}_n-\cA(n,k)\bar{z}_k &=\sum_{m=0}^{n-1} \cA(n,m) P^1_m(A_my_m+f_m(y_m+(\bbI-P^3_m)z_m)-y_{m+1}) \\
&\phantom{=}-\sum_{m=0}^{k-1} \cA(n,m) P^1_m(A_my_m+f_m(y_m+(\bbI-P^3_m)z_m)-y_{m+1}) \\
&\phantom{=}-\sum_{m=n}^{+\infty} \cA(n,m) P^2_m(A_my_m+f_m(y_m+(\bbI-P^3_m)z_m)-y_{m+1})\\
&\phantom{=}+\sum_{m=k}^{+\infty} \cA(n,m) P^2_m(A_my_m+f_m(y_m+(\bbI-P^3_m)z_m)-y_{m+1})
\end{split}
\]
and consequently 
\[
\begin{split}
\bar{z}_n-\cA(n,k)\bar{z}_k &=\sum_{m=k}^{n-1} \cA(n,m) P^1_m(A_my_m+f_m(y_m+(\bbI-P^3_m)z_m)-y_{m+1}) \\
&\phantom{=}+\sum_{m=k}^{n-1} \cA(n,m) P^2_m(A_my_m+f_m(y_m+(\bbI-P^3_m)z_m)-y_{m+1}).
\end{split}
\]
Therefore, 
\[
\bar{z}_{n+1}=A_n\bar{z}_n+\left(P^1_n+P_n^2\right)\left(A_ny_n+f_n(y_n+\bar{z}_n)-y_{n+1}\right)
\]
for every $n\in \bbN$. This easily implies that 
\begin{equation}\label{sol disc}
\begin{split}
y_{n+1}+\bar{z}_{n+1} &=A_n(y_n+\bar{z}_n)+f_n(y_n+\bar{z}_n) \\
&\phantom{=}-P^3_n(A_ny_n+f_n(y_n+\bar{z}_n)-y_{n+1}),
\end{split}
\end{equation}
for all $n\in \bbN$. We define $\mathbf{x}=(x_n)_{n\in \bbN}$ by
\[
x_n=y_n+\bar{z}_n, \quad n\in \bbN.
\]
Then, 
\[
P^3_n(x_n-y_n)=P^3_n\bar{z}_n=P^3_n(\bbI-P^3_n)z_n=0,
\]
for $n\in \bbN$. Moreover, using~\eqref{ln disc} we have that 
\[
|\bar{z}_n| =|(P^1_n+P^2_n)z_n|\le |P^1_nz_n|_n+|P^2_nz_n|_n \le |z_n|_n \le \|\mathbf{z}\|_\infty,
\]
for each $n\in \bbN$. Thus, 
\[
\sup_{n\in \bbN} |x_n-y_n| =\sup_{n\in \bbN} |\bar{z}_n| \le C\delta. 
\]
In addition, \eqref{sol disc} implies that $x_{n+1}-A_nx_n-f_n(x_n)\in \Ima P^3_n$ for all $n\in \bbN$. Finally, we have that 
\[
\begin{split}
 &|x_{n+1}-A_n x_n-f_n(x_n)| \\
&=|-P^3_n(A_ny_n+f_n(y_n+\bar{z}_n)-y_{n+1})| \\
&=|-P^3_n(A_ny_n+f_n(y_n+(\bbI-P^3_n)z_n)-y_{n+1})| \\
&=|z_n-\bar z_n| \\
&=|z_n-(\bbI-P^3_n)z_n |\\
&=|P_n^3 z_n| \\
&\le (2D+1)\nu_n^d |z_n| \\
&\le (2D+1)\nu_n^d|z_n|_n \\
&\le C\delta (2D+1)\nu_n^d,
\end{split}
\]
since 
\[
\begin{split}
|z_n| &\le |P_n^1z_n|+|P_n^2z_n|+|P_n^3z_n| \\
&\le |P_n^1z_n|_n+|P_n^2z_n|_n+\frac{\mu_{n+1}^\lambda-\mu_n^\lambda}{\mu_{n+1}^\lambda}|P_n^3z_n|_n \\
&\le |P_n^1z_n|_n+|P_n^2z_n|_n+|P_n^3z_n|_n  \\
&=|z_n|_n.
\end{split}
\]
Hence, \eqref{EST2} holds and the proof of the theorem is completed. 
\end{proof}

\begin{remark}
One can now easily  formulate and prove  the appropriate  versions of Corollaries~\ref{cor} and~\ref{corx} in the present setting.  
\end{remark}

\begin{remark}
As in the continuous time case, Theorem \ref{thm2} may be applied to several settings that we list below:
\begin{itemize}
\item the case of exponential dichotomies which corresponds to $\mu_n=e^n$, $\nu_n=1$ and $P_n^3=0$ for $n\in \mathbb N$;
\item the case of partial exponential dichotomies (see~\cite{BD2}) that corresponds to $\mu_n=e^n$ and $\nu_n=1$ for $n\in \mathbb N$;
\item the case of tempered exponential dichotomies that corresponds to $\mu_n=e^n$, $P_n^3=0$ for $n\in \mathbb N$ and with tempered $\nu$, i.e. 
\[
\lim_{n\to +\infty}\frac{1}{n}\ln \nu_n=0;
\]
\item the case of nonuniform  polynomial  dichotomies (see \cite{BS0}) which corresponds to $\mu_n=\nu_n=1+n$ and $P_n^3=0$ for $n\in \mathbb N$;
\item the case of $(\mu, \nu)$-dichotomies (see~\cite{BS2}) that corresponds to our general setting with $P_n^3=0$, $n\in \mathbb N$;
\end{itemize}
 Consequently, Theorem \ref{thm2} gives us an unified approach to the shadowing property in all these settings and, in particular, we can recover some of the results from \cite{BD0,BD2,BD1,BD3} and from references therein. On the other hand, we observe that it does not generalize all the results in \cite{BD0,BD2,BD1,BD3} because in those papers we studied a very general version of the shadowing property called $B$-Lipschitz shadowing property, where $B$ is some Banach sequence space. The size of a pseudotrajectory and its deviation from an exact trajectory is then measured via the norm on $B$. When $B=l^\infty$ is the space of all bounded sequences equipped with a supremum norm, we have the usual shadowing property as studied in present paper. However, if we take $B$ to be some other sequence space (for example $l^p$, $1\le p<\infty$), 
 we get different versions of the standard shadowing property that are not necessarily covered by Theorem \ref{thm2}. 

We end by remarking that, as in the case of continuous time, Thereom~\ref{thm2} is a first shadowing result that deals with the case when $(A_n)_n$ admits a polynomial, or more generally, a $(\mu, \nu)$-dichotomy.
\end{remark}


\medskip{\bf Acknowledgements.}
L.B. was partially supported by a CNPq-Brazil PQ fellowship under Grant No. 306484/2018-8. D.D. was supported in part by Croatian Science Foundation under the project IP-2019-04-1239 and by the University of Rijeka under the projects uniri-prirod-18-9 and uniri-pr-prirod-19-16.


\end{document}